\documentclass[12pt,a4paper]{amsart}
\usepackage{palatino,amsfonts,amssymb,enumerate,bm,hhline,makecell,array}
\usepackage{array,mathrsfs}
\usepackage[all,ps,cmtip]{xy} 
\usepackage[normalem]{ulem}
\usepackage{bm,mathtools}
\usepackage[pagebackref]{hyperref}
\usepackage[usenames,dvipsnames]{color}

\usepackage{enumitem}
\setlist[itemize]{noitemsep,nolistsep}
\setlist[enumerate]{noitemsep,nolistsep}
\usepackage{colortbl} 

\allowdisplaybreaks
\let\mathcal\mathscr

\def\C{{\bf C}}

\def\P{{\bf P}}

\def\phi{\varphi}

\def\cF{\mathcal{F}}

\def\cO{\mathcal{O}}
\def\cG{\mathcal{G}}

\def\cX{\mathcal{X}}

\def\lra{\longrightarrow}
\def\llra{\hbox to 10mm{\rightarrowfill}}
\def\lllra{\hbox to 15mm{\rightarrowfill}}

\def\llla{\hbox to 10mm{\leftarrowfill}}
\def\lllla{\hbox to 15mm{\leftarrowfill}}

\def\hra{\hookrightarrow}
\def\lhra{\ensuremath{\lhook\joinrel\relbar\joinrel\rightarrow}}

\def\isom{\simeq}

\def\k{{\C}}

\DeclareMathOperator{\codim}{codim}

\DeclareMathOperator{\Hom}{Hom}

\DeclareMathOperator{\length}{length}

\DeclareMathOperator{\Pic}{Pic}

\DeclareMathOperator{\pr}{{\mathsf{pr}}}

\def\llra{\hbox to 10mm{\rightarrowfill}}
\def\lllra{\hbox to 15mm{\rightarrowfill}}

\def\subset{\subseteq}

\newtheorem{lemm}{Lemma}

\newtheorem{theo}[lemm]{Theorem}

\newtheorem{coro}[lemm]{Corollary}
\newtheorem{prop}[lemm]{Proposition}

\theoremstyle{definition}

\newtheorem{rema}[lemm]{Remark}

\newtheorem{conj}[lemm]{Conjecture}

\theoremstyle{remark}
\newtheorem*{remark*}{Remark}
\newtheorem*{note*}{Note}

\def\rr{n}
\def\kk{r}

 \makeatletter
\@namedef{subjclassname@2020}{Mathematics Subject Classification (2020)}
\makeatother

\begin{document}
\title{On a conjecture of  Kazhdan and Polishchuk}

\author[O.\ Debarre]{Olivier Debarre}
\address{Universit\'e   Paris Cit\'e and Sorbonne Universit\'e, CNRS,   IMJ-PRG, F-75013 Paris, France}
 \email{{\tt olivier.debarre@imj-prg.fr}}


\thanks{This project has received funding from the European
Research Council (ERC) under the European
Union's Horizon 2020 research and innovation
programme (ERC-2020-SyG-854361-HyperK)}

  \begin{abstract}
We discuss a conjecture made by Alexander Polishchuk and David Kazhdan at the 2022 ICM about a variety naturally attached to any stable vector bundle of rank 2  and degree $2g- 1$ on a smooth projective complex curve of genus $g$.
  \end{abstract}


\maketitle 

\section{The conjecture}\label{s1}

 The aim of this note is to discuss   the following conjecture made by Alexander Polishchuk and David Kazhdan.\ It appeared in the survey paper~\cite{bk} as Conjecture~3.6.\ Everything is over   the field of complex numbers.

\begin{conj}\label{c2}
 Let $C$ be a smooth projective curve of genus $g\ge 2$ and let $E$ be a stable vector bundle on $C$ of rank 2  and degree $2g- 1$.\ 
Let 
\begin{equation}\label{deff}
\cF_E\coloneqq \{ ([L],s) \mid [L]\in\Pic^0(C),\ s\in \P(H^0(C,  E\otimes L))\}.
\end{equation}
Then,
\begin{itemize}
\item[(a)] $\cF_E$ is irreducible,
\item[(b)]  $\dim (\cF_E) = g$,
\item[(c)]  $\cF_E$ has rational singularities.
\end{itemize}
\end{conj}

When $E$ is a {\em general} stable vector bundle on $C$ of rank $2$  and degree~$2g-1$, the scheme~$\cF_E$ is smooth irreducible of dimension $g$ (Proposition~\ref{prop22}).\ The aim of this note is to prove items~(a) and~(b) above for all  stable $E$  and all  $g\ge2$.

\begin{theo}\label{mth}
 Let $C$ be a smooth projective curve of genus $g\ge 2$ and let $E$ be a stable vector bundle on $C$ of rank~$2$  and degree $2g- 1$.\ 
Then
 $\cF_E $
  is irreducible, normal, and a local complete intersection of dimension~$g$.
 \end{theo}
 

For any   $\rr\ge 3$ and $g\ge2$, we  construct, on any  hyperelliptic  curve  of genus $g$, stable vector bundles $E$   of rank $\rr$  and degree $\rr (g-1)+ 1$  for which the scheme $\cF_E $
  is  reducible (see Remark~\ref{r23}).\ So the natural extension of the conjecture to stable  vector bundles of higher ranks fails in general (although it holds for $E$ general by Proposition~\ref{prop22}).
  
 After a first version of this note was written, I sent it to Alexander Braverman, who directed me to the eprint~\cite{kapo}, where some  progress is made on   Conjecture~\ref{c2}   (called Conjecture~B in that paper).\ In particular, the   conjecture is proved there for any curve of genus~$2$ (\cite[Theorem~4.20]{kapo}) and any nonhyperelliptic curve of genus 3 (\cite[Theorem~4.34]{kapo}).\ This first version of the present note also contained a proof of the conjecture for~$g=2$ but, since it was essentially the same as in \cite[Theorem~4.20]{kapo}, I decided   to include instead in this second version a slightly different proof that makes use of Theorem~\ref{mth}.
 
     Conjecture~C(i) in \cite{kapo} is our Lemma~\ref{l32} (proved by Brivio and Verra) and   Conjecture~C(ii) is our Lemma~\ref{l34}.\   Conjecture~C(ii) implies, by  \cite[Lemma~4.9]{kapo}, the  normality of~$\cF_E$.

\section{Properties of the scheme $\cF_E$ and the conjecture for general $E$}\label{s2}

In this section,  $C$ is  a smooth projective curve   of genus $g\ge 2$ and  $E$ is a  vector bundle on $C$ of positive rank $\rr $  and degree $d\coloneqq \rr (g-1)+ 1$.\ The set~$\cF_E$ was defined in~\eqref{deff}.
 
\subsection{General  properties of the scheme $\cF_E$}\label{se21}
 For all $[L]\in \Pic^0(C)$, one has $h^0(C,  E\otimes L)\ge \chi(C,  E\otimes L)=\chi(C,E)=1$, hence the first projection
$$\pr_1\colon \cF_E \lra \Pic^0(C) $$
is dominant.\ Its fibers are projective spaces, hence  $\cF_E$ is connected.\

Using a standard argument, we define a scheme structure on~$\cF_E$  and show that it has dimension at least $g$ at every point.  

\begin{prop}\label{prop31}
  Let $C$ be a smooth projective curve of genus $g\ge 2$ and let $E$ be a  vector bundle on~$C$ of rank $\rr $  and degree $\rr (g-1)+ 1$.\ The scheme $\cF_E$ has everywhere dimension at least $g$ and it is a local complete intersection at every point where its dimension is exactly $g$.
\end{prop}

\begin{proof}
Let $D$ be an effective divisor on $C$ of  degree $  g-1$.\ By Serre  duality, we have 
$$H^1(C,   E\otimes L(D))\isom H^0(C,  \omega_C\otimes E^\vee \otimes L^\vee(-D))^\vee=\Hom(L(D) ,\omega_C\otimes E^\vee)$$ and the latter space vanishes  for all $[L]\in\Pic^0(C)$  by stability of $\omega_C\otimes E^\vee$.\ When $[L]$ varies in~$\Pic^0(C)$, the vector spaces $H^0(C,E\otimes L(D))$ are then the fibers of a vector bundle~$\cF$ on~$\Pic^0(C)$ of rank $\chi(C, E(D))=1+\rr (g-1)$, whereas 
the vector spaces~$H^0(C,  E\otimes L(D)\vert_D)$ are the fibers of a vector bundle $\cG$ on $\Pic^0(C)$ of rank $m\coloneqq \rr (g-1)$.\ The kernel at a point $[L]$ of the  restriction morphism
\begin{equation}\label{defalpha}
\alpha\colon \cF\lra \cG
\end{equation}
is $H^0(C,  E\otimes L)$.\ The vector bundles $\cF$ and $\cG$ have  respective ranks $m+1$ and 
 $m$.

 Set  
$$\cX\coloneqq \{ (u,[x]) \in \Hom(\k^{m+1},\k^{m})\times \P(\k^{m+1})\mid u(x)=0 \}.$$
Looking at the projection $\cX\to \P(\k^{m+1})$, we see that $\cX$ is smooth of codimension $m$ in $\Hom(\k^{m+1},\k^{m})\times \P(\k^{m+1})$.
 
The variety $\Pic^0(C)$ is covered by open subsets $U$ with morphisms  to $\Hom(\k^{m+1},\k^{m})$ such that $$\pr_1^{-1}( U)=U\times_{\Hom(\k^{m+1},\k^{m})}\cX.$$
This defines a scheme structure on~$\cF_E$   that makes it into  a subscheme of $\P(\cF)$.

Since $\cF_E$ is not empty, it  follows that it has everywhere codimension at most~$m$ in~$\P(\cF)$ and that it is a local complete intersection at every point where its codimension is exactly~$m$.
\end{proof}

\subsection{Brill--Noether loci}\label{bnl}

Let $\kk $ be a nonnegative integer.\ The Brill--Noether locus
\begin{equation}\label{defbk}
B_\kk (E)\coloneqq \{ [L]\in\Pic^0(C)\mid h^0(C, E\otimes L )\ge \kk +1\}
\end{equation}
 is the degeneracy subscheme of $\Pic^0(C) $ where the morphism $\alpha\colon \cF\to \cG$ defined in~\eqref{defalpha} has corank $\ge \kk $; therefore, it is everywhere of codimension $\le \kk (\kk +1)$ in $ \Pic^0(C)$.\ These loci were studied by Ghione in~\cite{ghi} and Lazarsfeld in \cite[\S~2]{laz}: they prove  (in our situation) that~$B_\kk (E)$ is nonempty when $\kk (\kk +1)\le g$,  connected when $\kk (\kk +1)< g$, and that when $B_\kk (E)$ is everywhere of codimension $ \kk (\kk +1)$ in $ \Pic^0(C)$,   it is Cohen--Macaulay and its cohomology class is
 $$\Bigl( \prod_{i=0}^{\kk }\frac{i!}{(\kk +i)!}\Bigr)\ \rr ^{\kk (\kk +1)}\ \theta^{\kk (\kk +1)},$$
 where $\theta$ is the canonical principal polarization on $ \Pic^0(C)$.

\subsection{The conjecture for general vector bundles} 
Stable vector bundles $E$ considered here have  a fine moduli space $U_C(\rr ,d)$ (since $\rr $ and~$d$ are relatively prime, stability and semistability are equivalent). This is a smooth projective variety of dimension $\rr ^2(g-1)+1$.\ When one fixes the determinant, one obtains a smooth projective   irreducible fine moduli space $SU_C(\rr ,d)\subset  U_C(\rr ,d)$.

 For pairs $(E,V)$, where $E$ is a vector bundle   on $C$ and $V\subset H^0(C,E)$ a one-dimensional space of sections of~$E$, Bradlow, Thaddeus (\cite{tha}), Le Potier (\cite{lp}), and  others introduced a notion of stability depending on a real parameter $\sigma$ (\cite[Definition~2.1]{BGMN}).\  Without going into details, let us just mention that in our situation (where $\rr $ and $d$ are relatively prime), for   $\sigma$ positive and small enough, stability of the pair $(E,V)$ and  stability of the bundle $E$ are   equivalent, hence are independent of $\sigma$ (\cite[Proposition~2.5]{BGMN}).\ The corresponding moduli space for  stable pairs $(E,V)$, which we denote by $G_C(\rr ,d)$,\footnote{It is denoted by $G_0(\rr ,d,1)$, or simply $G_0$, in \cite[Section~2.1]{BGMN}.} is then smooth irreducible of  dimension $\rr ^2(g-1)+1$ (\cite[Theorem~7.1]{BGMN}).\ The forgetful map 
\begin{equation*}
\rho\colon G_C(\rr ,d)  \lra U_C(\rr ,d)
\end{equation*}
is a  dominant birational morphism.

The next easy proposition shows that the natural extension of Conjecture~\ref{c2}  to the case of stable vector bundles   of any    rank $\rr >0$  and degree $\rr (g-1)+ 1$   holds when the vector bundle $E$ is {\em general}  in the moduli space $U_C(\rr ,\rr (g-1)+ 1)$.

\begin{prop}\label{prop22}
Let $C$ be a smooth projective curve of genus $g\ge 2$ and let $E$ be a {\em general} stable vector bundle on $C$ of rank $\rr $  and degree $\rr (g-1)+ 1$.\ The scheme $\cF_E$ is smooth irreducible of dimension $g$.
\end{prop}

\begin{proof}
 The tensor product map
$$SU_C(\rr ,d)\times \Pic^0(C)\lra U_C(\rr ,d)$$
is \'etale Galois.\ Since $ G_C(\rr ,d)$ is smooth of dimension $\rr ^2(g-1)+1$,  the fiber product 
$$(SU_C(\rr ,d)\times \Pic^0(C))\times_{U_C(\rr ,d)} G_C(\rr ,d)$$
is also smooth of dimension $\rr ^2(g-1)+1$. It maps onto  $SU_C(\rr ,d)$ by projection and the fiber of a point $[E]$ is exactly $\cF_E$.\  By generic smoothness,~$\cF_E$ is smooth  of dimension $g$ when $[E]$ is general.\ Being connected, as we saw above, it is also irreducible.
\end{proof}

\begin{rema}\label{r23}\upshape
  Let $C$ be any smooth projective curve of genus $g\ge 2$.\ For every $\rr \ge g+1$, one can explicitly construct    stable vector bundles $E$ on $C$ of rank $\rr $  and degree $\rr (g-1)+ 1$ for which~$\cF_E$ is reducible, of dimension $\ge \rr -1$.\ When $C$ is hyperelliptic, one can even construct such an $E$   of rank 
  $\rr \ge 3$.\ So the natural extension of Conjecture~\ref{c2} fails in ranks $\ge 3$.
\end{rema}

\section{Proof of Theorem~\ref{mth}}\label{s3}

In this section,  we fix a stable  vector bundle $E$  of  rank~$2$ and degree~$2g-1$ on a smooth projective  curve~$C$ of genus $g$.\ The scheme $\cF_E$ was defined in Section~\ref{se21} and the Brill--Noether loci $B_\kk (E)\subset \Pic^0(C) $ in~\eqref{defbk}.\ The  first projection
\begin{equation*}
\pr_1\colon \cF_E \lra \Pic^0(C) 
\end{equation*}
is dominant, with fibers   $\P^{\kk }$ over $B_\kk (E)\smallsetminus B_{\kk +1}(E)$.\ A result of Raynaud (\cite[Proposition~1.6.2]{ray}) says that  $h^0(C,  E\otimes L) =1$ when  $[L]$ is general in $\Pic^0(C)$.\ In other words,
\begin{equation}\label{bo}
\codim_{\Pic^0(C)} (B_\kk (E) )>0
\end{equation}
for all $\kk \ge 1$.\ Thus, only one irreducible component of  $\cF_E$  dominates~$\Pic^0(C)$ via $\pr_1$ and it has the (expected) dimension $g$.\ Items (a) and (b) of Conjecture~\ref{c2} are then equivalent to the inequalities
\begin{equation}\label{bk}
\codim_{\Pic^0(C)} (B_\kk (E) )\ge \kk +1
\end{equation}
for all $\kk \ge 1$.\ 
The next lemma shows that it is enough to prove them for  $\kk =1$ and $\kk =2$.

\begin{lemm}\label{l31}
For each $\kk \ge 2$, there is a set-theoretic inclusion 
$$B_\kk (E)+C-C\subset B_{\kk -2}(E).$$
 In particular, when $\kk \ge 3$, one has $\dim(B_\kk (E))\le \dim( B_{\kk -2}(E))-2$.
\end{lemm}

\begin{proof}
The inclusion is easy: if $[L]\in B_\kk (E)$, one has $h^0(C,E\otimes L)\ge \kk +1$, hence, for all $p,q\in C$, we obtain $h^0(C,E\otimes L(p-q))\ge h^0(C,E\otimes L(-q))\ge \kk -1$ (since $E$ has rank $2$).\ In other words, $B_\kk (E)+C-C$ is contained in $ B_{\kk -2}(E)$.

When $\kk\ge 3$, the inequality $\dim(B_\kk (E))\le \dim( B_{\kk -2}(E))-2$ then follows from the fact (proved in \cite[corollaire~2.7]{deb}) that since the curve $C-p$ generates $\Pic^0(C)$, either the sum $B_\kk (E)+C-C$, hence also $B_{\kk -2}(E)$, is   equal to $\Pic^0(C)$ (which, since $\kk \ge3$, is impossible by~\eqref{bo}), or it has dimension $\dim(B_\kk (E))+2$. 
\end{proof}

We now treat the case $\kk =1$.

 \begin{lemm}\label{l32}
The scheme $B_1(E)$ is nonempty and has everywhere codimension $2$ in $ \Pic^0(C)$.\ It is  Cohen--Macaulay, with class $2\theta^2$.
\end{lemm}

\begin{proof}
The inequality~\eqref{bk} for $\kk =1$ was   proved  in  \cite[Theorem~3.5.1]{bve} (their locus $E_\xi$ is   our~$B_1(\xi)$).\ The rest of the statement then follows from Remark~\ref{bnl}.
\end{proof}

The next lemma deals with the remaining case $\kk =2$.

 \begin{lemm}\label{l33}
One has $\codim_{\Pic^0(C)} B_2(E) \ge 3$.
\end{lemm}

\begin{proof}
Assume to the contrary that  $B_2(E)$ contains an irreducible subvariety $B$ of dimension~$g-2$.\ Fix $p_0\in C$.\ As explained in the proof of Lemma~\ref{l31}, since the curve~$C-p_0$ generates~$\Pic^0(C)$, the image of the  difference map 
$$  B\times C  \lra \Pic^0(C)\ ,\qquad  ([L],q)\longmapsto [L(p_0-q)]
$$ is an irreducible hypersurface $H\subset \Pic^0(C)$.\ For the same reason, the sum map 
$$   H\times C  \to \Pic^0(C)\ ,\qquad  ([N],p)\longmapsto [N(p-p_0)]$$
 is generically finite, say of degree $e$.\ For~$[M]$ general in $\Pic^0(C) $, one can therefore write 
$$M=N_1(p_1-p_0)=\cdots= N_e(p_e-p_0)$$
with $[N_1],\dots,[N_e]\in H$, and $p_1,\dots,p_e\in C$ pairwise distinct, and in turn $N_i = L_i(p_0-q_i)$, with $[L_1],\dots,[L_e]\in B$ and $ q_1,\dots,q_e\in C$.\ For each 
 $i\in\{1,\dots,e\}$, we have $M=L_i(p_i-q_i)$ and a chain of inequalities
$$h^0(C,  E\otimes M)\ge h^0(C,  E\otimes M(-p_i))=h^0(C,E\otimes L_i(-q_i))\ge h^0(C,E\otimes L_i)-2.
$$
Since $h^0(C,  E\otimes M) =1$  by Raynaud's result, and $h^0(C,  E\otimes L_i) \ge3 $ because $[L_i]\in B_2(E)$, all these numbers are equal to $1$ and 
  the ``unique'' section of $E\otimes M$ vanishes at the distinct points $p_1,\dots,p_e$, so that there is an injection
\begin{equation}\label{eqinj}
\cO_C( p_1+\dots+p_e)\lhra E\otimes M.
\end{equation}

 By \cite[Theorem~1, p.~95]{muke}, there is an abelian quotient $\pi\colon \Pic^0(C)\to Q$ and an ample  irreducible hypersurface $H'\subset Q$ such that $H=\pi^* H'$.\ By the projection formula, one then has
\begin{equation}\label{eq1}
e=H\cdot C=\pi^*H'\cdot C=H'\cdot \pi_*C \ge \dim(Q),
\end{equation}
where the   inequality holds because the curve $\pi(C)$ spans $Q$ and $H'$ is ample (see for example \cite[Section~4]{debd}).

The irreducible hypersurface  $H$  is stable by translation by the kernel $K$ of~$\pi$.\ For $b$ general in $B$ and~$p$ general in $C$, the difference map $\tau\colon B\times C  \to H$ is \'etale at $(b,p)$, hence  $T_{H,b+p}=T_{B,b}+T_{C,p}$ (all tangent spaces are translated in $T\coloneqq T_{\Pic^0(C),0}$) and this vector space contains $T_{K,0}$.\ This means that the image of $T_{K,0}$ in the 2-dimensional vector space  $T/T_{B,b}$ is contained in the image of the line $T_{C,p}$ for all $p$ general in~$C$.\ It must therefore be $0$, which means that~$T_{K,0}$ is contained in~$T_{B,b}$ for all $b$ general in $B$.\ By \cite[lemme~2.3]{deb}, this implies that~$B$ is stable by translation by the connected component $K^0$ of $0$ in $K$.

Fix again    $[M]$ general in $\Pic^0(C) $ and keep the notation of the beginning of the proof.\ For~$[N]$  general in  $K^0$, one has $M\otimes N=(L_i\otimes N)(p_i-q_i)$, with $[L_i\otimes N]\in B$ (because $B+K^0=B$, as we just proved), so, as  in~\eqref{eqinj}, there is an injection $\cO_C( p_1+\dots+p_e)\hra E\otimes M \otimes N$.\ In other words, we have
$$h^0(C, E\otimes M\otimes (N^\vee( p_1+\dots+p_e))^\vee)\ge 1
$$
for $N^\vee( p_1+\dots+p_e) $ general in $K^0+p_1+\dots+p_e\subset \Pic^e(C)$, hence   for all elements of that  subset.
By \cite[Proposition~3.4.1]{bve}, this implies  $\dim(K^0) \le g-1-e$, that is,
$$e\le g-1-\dim(K^0)=\dim(Q)-1 ,
$$ 
which contradicts the inequality~\eqref{eq1}.\end{proof}

The scheme $\cF_E$ is therefore irreducible of dimension $g$ and, by Proposition~\ref{prop31}, it is  a local complete intersection.\ It remains to show that it is nonsingular in codimension $1$.\ Let~$r,d $ be nonnegative integers.\ We set $M_E\coloneqq \omega_C\otimes \det(E)^\vee\in \Pic^{-1}(C)$ and 
$$B_{r,d}(E)\coloneqq \{ ([L],D) \in \Pic^0(C) \times C^{(d)}\mid h^0(C,M_E\otimes (L^\vee)^{\otimes 2}(D))>0,\ h^0(C,E\otimes L(-D))\ge r+1\}.
$$

By  \cite[Lemma~4.9]{kapo}, it is enough to show the following result.

\begin{lemm}\label{l34}
For  any $d\in\{1,\dots, g-1\}$ and $r\ge 0$, the locus
$ B_{r,d}(E) $
has dimension $\le g-r-2$. 
\end{lemm}

\begin{proof}
The map $([L],D) \mapsto ([L^\vee(D)],D)$ induces an     isomorphism between $B_{r,d}(E)$ and
$$B'_{r,d}(E)\coloneqq \{ ([M],D) \in \Pic^{d}(C) \times C^{(d)}\mid  h^0(C,M_E\otimes  M^{\otimes {2}}(-D))>0,  h^0(C,E\otimes M^\vee)\ge r+1\}.$$
Let $ B''_{r,d}(E)\subset \Pic^{d}(C)$ be the image of the  first projection of $B'_{r,d}(E)$.\  

If $r=0$, \cite[Proposition~3.4.1]{bve} implies $\dim(B''_{r,d}(E))\le g-1-d=g-r-1-d$.\ 

If $r\ge1$, the  image of the map   
$$B''_{r,d}(E)\times C^{(d)}\to \Pic^0(C),\qquad ([M],D')\mapsto [M^\vee(D')]$$  is contained in~$B_{r}(E)$, hence, since $r\ge 1$, it has dimension $\le g-r-1$ by~\eqref{bk}.\ As explained in the proof of Lemma~\ref{l31}, this implies again $\dim ( B''_{r,d}(E))\le g-r-1-d$.\ 

Finally, assume that the fiber   of some~$[M]\in B''_{r,d}(E)$ for the projection $B'_{r,d}(E)\to B''_{r,d}(E)$ has dimension $d$.\ It is then equal to $ C^{(d)}$.\ This means that 
 $$\forall  D\in C^{(d)}\qquad h^0(C,M_E\otimes M^{\otimes {2}}(-D))>0
$$
 hence, $h^0(C,M_E\otimes M^{\otimes 2} ) \ge d+1$.\ But $M_E\otimes M^{\otimes 2}$ has degree $ 2d-1\le 2g-3$, so this contradicts Clifford's theorem.\ Therefore, all fibers have dimension $\le d-1$  hence, $\dim (B_{r,d}(E)) \le g-r-1-d+d-1=g-r-2 $, as desired.
 \end{proof}

When $g=2$, it is   easy to deduce Conjecture~\ref{c2} from Theorem~\ref{mth}  (this case of the conjecture was already proved in \cite[Theorem~4.20]{kapo}, with a slightly different proof). 
 
 \begin{coro}
Conjecture~\ref{c2} holds when $g=2$.
\end{coro}

\begin{proof}
 Let $C$ be a smooth projective curve of genus $  2$ and let $E$ be a stable vector bundle on~$C$ of rank~$2$  and degree $3$.\ By Theorem~\ref{mth}, the scheme $\cF_E$ is a normal irreducible surface and, as we saw in Section~\ref{s3}, $B_2(E)$ is empty and $B_1(E)$ is a finite subscheme of $\Pic^0(C)$ locally defined by two equations.\ A local analysis based on the proof of Proposition~\ref{prop31} shows that~$\cF_E$ is the blowup of the smooth surface $\Pic^0(C)$ along its finite subscheme $B_1(E)$.\footnote{More generally, for any $g\ge 2$, by~\cite[Proposition~4.1]{kapo} and~\eqref{bk},  the scheme $\cF_E$ is the blowup of~$\Pic^0(C)$ along its   codimension $2$ subscheme $B_1(E)$.}\ If~$B_1(E)$ is defined locally around a point $[L]$ by two equations $a(z)=b(z)=0$, this blowup is defined by $ua(z)+vb(z)=0$, where $(u,v)$ is in the exceptional curve $\P^1_L$.\ If $T_L$ is the  Zariski tangent space to   $B_1(E)$  at $[L]$, one sees that
\begin{itemize}
\item  either $\dim(T_L)=0$ (the   scheme   $B_1(E)$ is smooth   at $[L]$), in which case the surface $\cF_E$ is smooth   along $\P^1_L$;
\item or  $\dim(T_L)=1$, in which case  the surface~$\cF_E$ is singular at exactly one point of   $\P^1_L$, where it has a singularity of type $A_{m}$, where $m=\length_{[L]}(B_1(E))-1$;
\item or  $\dim(T_L)=2$, in which case  the surface  $\cF_E$ is singular along the whole of $\P^1_L$.
\end{itemize}
 Since $\cF_E$ is normal by Theorem~\ref{mth}, the last   case cannot happen and the conjecture is proved.
\end{proof}

\begin{rema}[Brill--Noether loci of maximal dimensions]\upshape
 The inequalities~\eqref{bk} are optimal: on any  hyperelliptic curve of genus at least $  2$  and on any curve  of genus 3, one can construct   a   stable vector bundle~$E$ of rank~$2$ and degree~$2g-1$  such that
$B_\kk (E)$ has dimension $  g-\kk -1$ for all $\kk \in\{1,\dots, g-1\}$.\ However, when  the Clifford index of $C$ is at least $2$ (in particular, one has~$g\ge5$),   the set~$B_{g-1}(E)$ is empty by \cite[Proposition~3.5]{lan}.
\end{rema}

\begin{rema}[Nonreduced Brill--Noether loci]\upshape
On any smooth projective curve $C$ of genus at least~$ 2$, one can construct stable   vector bundles $E$ of rank~$2$ and degree~$2g-1$
for which the scheme $B_1(E)$ has a nonreduced component (of codimension $2$ in $\Pic^0(C)$).\ The scheme  $\cF_E$ is then singular in codimension $2$.
\end{rema}

\end{document}